\batchmode
\makeatletter
\def\input@path{{"/Users/russw/Documents/Research/mypapers/A short proof of the Hilton-Milner theorem/"}}
\makeatother
\documentclass[12pt,oneside,english,lowtilde]{amsart}
\usepackage[T1]{fontenc}
\usepackage[latin9]{inputenc}
\usepackage{babel}
\usepackage{url}
\usepackage{amstext}
\usepackage{amsthm}
\usepackage{amssymb}
\usepackage{geometry}
\usepackage[dvips,pdfusetitle,
 bookmarks=true,bookmarksnumbered=false,bookmarksopen=false,
 breaklinks=false,pdfborder={0 0 1},backref=false,colorlinks=false]
 {hyperref}
\usepackage{breakurl}

\makeatletter
\theoremstyle{plain}
\newtheorem{thm}{\protect\theoremname}
\theoremstyle{plain}
\newtheorem{lem}[thm]{\protect\lemmaname}
\theoremstyle{remark}
\newtheorem{rem}[thm]{\protect\remarkname}

\usepackage{lmodern}

\makeatother

\providecommand{\lemmaname}{Lemma}
\providecommand{\remarkname}{Remark}
\providecommand{\theoremname}{Theorem}

\begin{document}
\global\long\def\link{\operatorname{link}}%

\global\long\def\del{\operatorname{del}}%

\global\long\def\cone{\operatorname{Cone}}%

\global\long\def\depth{\operatorname{depth}}%

\global\long\def\shift{\operatorname{Shift}}%

\global\long\def\symdiff{\ominus}%

\global\long\def\shiftext{\shift^{\wedge}}%

\global\long\def\extalg{\bigwedge}%

\global\long\def\ff{\mathbb{F}}%

\global\long\def\bdry{\partial}%

\renewcommand{\proofname}{\hskip-\labelsep\spacefactor3000 }
\title{A short proof of the Hilton-Milner Theorem}
\author{Denys Bulavka and Russ Woodroofe}
\thanks{Work of the first author is partially supported by the Israel Science
Foundation grant ISF-2480/20 and the AARMS postdoctoral fellowship.
Work of the second author is supported in part by the Slovenian Research
Agency (research program P1-0285 and research projects J1-9108, J1-2451,
J1-3003, and J1-50000).}
\address{Einstein Institute of Mathematics, Hebrew University, Jerusalem 91904,
Israel}
\curraddr{Department of Mathematics \& Statistics, Dalhousie University, 6297
Castine Way, PO BOX 15000, Halifax, NS, Canada, B3H 4R2}
\email{denys.bulavka@dal.ca}
\urladdr{\url{https://kam.mff.cuni.cz/~dbulavka/}}
\address{Univerza na Primorskem, Glagoljaška 8, 6000 Koper, Slovenia}
\email{russ.woodroofe@famnit.upr.si}
\urladdr{\url{https://osebje.famnit.upr.si/~russ.woodroofe/}}
\begin{abstract}
We give a short and relatively elementary proof of the Hilton-Milner
Theorem.
\end{abstract}

\maketitle
The Hilton-Milner Theorem gives the maximum size of a uniform pairwise-intersecting
family of sets that do not share a common element.
\begin{thm}[Hilton and Milner 1967 \cite{Hilton/Milner:1967}]
\label{thm:HM}Let $k\leq n/2$. If $\mathcal{F}$ is a family of
pairwise-intersecting $k$-element subsets of $[n]$, where $\bigcap_{F\in\mathcal{F}}F=\emptyset$,
then $\left|\mathcal{F}\right|\leq{n-1 \choose k-1}-{n-1-k \choose k-1}+1$.
\end{thm}

In the current article, we will show Theorem~\ref{thm:HM} to follow
quickly from the following theorem, which we believe to be of some
independent interest. Two set systems $\mathcal{A}$ and $\mathcal{B}$
are \emph{cross-intersecting} if for every $A\in\mathcal{A},B\in\mathcal{B}$,
the intersection $A\cap B$ is nonempty. The \emph{shadow} $\bdry B$
of a $k$-element set $B$ consists of all the $(k-1)$-element subsets
of $B$; the shadow of a uniform set family is the union of the shadows
of its constituent sets, thus consists of all $(k-1)$-element subsets
of constituent sets.
\begin{thm}
\label{thm:MainTechnical}Let $2k-1\leq n$, let $\mathcal{A}$ be
a family of $(k-1)$-element subsets of $[n]$, and let $\mathcal{B}$
be a family of $k$-element subsets of $[n]$. If $\mathcal{A},\mathcal{B}$
are cross-intersecting, $\mathcal{B}$ is nonempty, and $\bdry\mathcal{B}\subseteq\mathcal{A}$,
then 
\[
\left|\mathcal{A}\right|+\left|\mathcal{B}\right|\leq{n \choose k-1}-{n-k \choose k-1}+1.
\]
\end{thm}

Note that the bound of Theorem~\ref{thm:HM} is attained with a single
$k$-element set $B$ that does not contain $1$, together with all
the $k$-element sets that contain $1$ and intersect $B$; the bound
of Theorem~\ref{thm:MainTechnical} is attained with a single $k$-element
set and all $(k-1)$-element sets that intersect it.

Our proof of Theorem~\ref{thm:HM} may be viewed as injective. Other
recent proofs of Theorem~\ref{thm:HM} were given in \cite{Frankl:2019,Hurlbert/Kamat:2018},
but instead of relying on a simple cross-intersecting type theorem,
both of these proofs rely on a certain ``partial complement'' operation.
In somewhat older work \cite{Frankl/Tokushige:1992} (see also \cite{Frankl:2016}),
Frankl and Tokushige gave a proof of Hilton-Milner from a different
cross-intersection theorem, but the proof is less elementary than
that of Theorem~\ref{thm:MainTechnical}, requiring the Schützenberger-Kruskal-Katona
Theorem. A recent preprint of Wu, Li, Feng, Liu and Yu \cite{Wu/Li/Feng/Liu/Yu:2026}
(now published) gives a proof based on still another cross-intersecting
theorem, but the existing proofs of this underlying result also seem
to be somewhat more difficult than our approach.

\subsection*{Shifting}

We recall that a system $\mathcal{F}$ of subsets of $[n]$ is \emph{shifted}
if for each $i<j$, whenever $F\in\mathcal{F}$, $j\in F$ and $i\notin F$,
then also $\left(F\setminus j\right)\cup i\in\mathcal{F}$. Here,
we abuse notation to identify $i,j$ with the singleton subsets $\{i\},\{j\}$
where it causes no confusion. The \emph{(combinatorial)} \emph{shifting
operation} $\shift_{i\leftarrow j}$ is defined as 
\begin{align*}
\shift_{i\leftarrow j}\mathcal{F}= & \left\{ F\in\mathcal{F}:j\notin F\text{ or }i\in F\text{ or }\left(F\setminus j\right)\cup i\in\mathcal{F}\right\} \\
 & \cup\left\{ \left(F\setminus j\right)\cup i:F\in\mathcal{F}\text{ is such that }j\in F,i\notin F\right\} .
\end{align*}

It is well-known that repeated applications of $\shift_{i\leftarrow j}$
over $i<j$ will eventually reduce an arbitrary set system to a shifted
set system, that the operation preserves the cross-intersecting property,
and that $\bdry\shift_{i\leftarrow j}\mathcal{F}\subseteq\shift_{i\leftarrow j}\bdry\mathcal{F}$
\cite{Frankl:1987,Frankl:1991,Gerbner/Patkos:2019,Herzog/Hibi:2011}.

\subsection*{Proof of Theorem~\ref{thm:MainTechnical}}

We carry out a straightforward induction on $n$.
\begin{proof}
If $n=2k-1$, then the upper bound is ${n \choose k-1}$, and the
result follows by noticing that if a $(k-1)$-element set is in $\mathcal{A}$,
then its complement cannot be in $\mathcal{B}$ (and vice-versa).

For the inductive step, we may assume that $\mathcal{A}$ and $\mathcal{B}$
are shifted; otherwise, shift. Let $\mathcal{A}(\neg n),\mathcal{B}(\neg n)$
consist of the subsets in $\mathcal{A},\mathcal{B}$ (respectively)
that do not contain $n$. It is immediate that $\mathcal{A}(\neg n),\mathcal{B}(\neg n)$
are shifted, cross-intersecting, and satisfy the shadow condition.
Let $\mathcal{A}(n),\mathcal{B}(n)$ be obtained by taking the families
consisting of the subsets in $\mathcal{A},\mathcal{B}$ that contain
$n$, then deleting $n$ from each subset. It follows quickly from
definitions that $\mathcal{A}(n),\mathcal{B}(n)$ are shifted, cross-intersecting,
and satisfy the shadow condition. 

As $\mathcal{A}$ and $\mathcal{B}$ are shifted, so $\mathcal{A}(\neg n)$
and $\mathcal{B}(\neg n)$ are nonempty, and hence by induction 
\begin{equation}
\left|\mathcal{A}(\neg n)\right|+\left|\mathcal{B}(\neg n)\right|\leq{n-1 \choose k-1}-{n-1-k \choose k-1}+1.\label{eq:notN}
\end{equation}

For $\mathcal{A}(n),\mathcal{B}(n)$, there are a few easy cases:

If $\mathcal{A}(n)$ is empty, then (by the shadow condition) also
$\mathcal{B}(n)$ is empty.

If $\mathcal{B}(n)$ is empty, then since $\mathcal{B}$ is nonempty
and shifted, we have $\{1,\dots,k\}\in\mathcal{B}$. Since every set
in $\mathcal{A}(n)$ intersects with $\{1,\dots,k\}$, we get $\left|\mathcal{A}(n)\right|+\left|\mathcal{B}(n)\right|=\left|\mathcal{A}(n)\right|\leq{n-1 \choose k-2}-{n-1-k \choose k-2}$.

If $\mathcal{B}(n)$ is nonempty, then by induction it holds that
\begin{align}
\left|\mathcal{A}(n)\right|+\left|\mathcal{B}(n)\right| & \leq{n-1 \choose k-2}-{n-1-(k-1) \choose k-2}+1\nonumber \\
 & \leq{n-1 \choose k-2}-{n-1-k \choose k-2}.\label{eq:BnNonempty}
\end{align}
 The result now follows from (\ref{eq:notN}), the bound on $\left|\mathcal{A}(n)\right|+\left|\mathcal{B}(n)\right|$,
and the Pascal's Triangle identity.
\end{proof}

\subsection*{Proof of Theorem~\ref{thm:HM}}

We will use the following lemma of Frankl and Füredi:
\begin{lem}[essentially Frankl and Füredi \cite{Frankl/Furedi:1986}]
\label{lem:Frankl-Furedi}If $\mathcal{F}$ is a pairwise-intersecting
family of $k$-element subsets of $[n]$ with $\bigcap_{F\in\mathcal{F}}F=\emptyset$,
then there is a shifted family $\mathcal{F}'$ satisfying the same
properties and with $\left|\mathcal{F}'\right|\geq\left|\mathcal{F}\right|$.
\end{lem}

\begin{proof}
Given the lemma, the proof of Theorem~\ref{thm:HM} is nearly immediate.
Let $\mathcal{F}$ be a shifted family satisfying the conditions of
the theorem. Define 
\begin{alignat*}{2}
\mathcal{A}= & \,\{F\setminus1\,\, & :F\in\mathcal{F}\text{ with }1\in\mathcal{F}\}\\
\mathcal{B}= & \,\{F & :F\in\mathcal{F}\text{ with }1\notin\mathcal{F}\} & .
\end{alignat*}
Since $\mathcal{F}$ is shifted, if $F\in\mathcal{F}$ does not have
$1$, then $\left(F\setminus i\right)\cup1\in\mathcal{F}$ for each
$i\in\mathcal{F}$. It follows that $\bdry\mathcal{B}\subseteq\mathcal{A}$.
Since $\mathcal{F}$ is intersecting, also $\mathcal{A},\mathcal{B}$
are cross-intersecting systems of subsets of $\{2,\dots,n\}$. Since
$\mathcal{F}$ has empty intersection, both of $\mathcal{A},\mathcal{B}$
are nonempty. The desired bound is now immediate from Theorem~\ref{thm:MainTechnical}.
\end{proof}
\begin{rem}
This proof requires only the special case of Theorem~\ref{thm:MainTechnical}
where the set systems are shifted.
\end{rem}

\subsection*{Proof of Lemma~\ref{lem:Frankl-Furedi}}

For completeness, we also prove the lemma.
\begin{proof}
Given $\mathcal{F}$ as in Theorem~\ref{thm:HM}, apply shifting
operations $\shift_{i\leftarrow j}$. Each such operation preserves
the pairwise-intersecting property and cardinality, but may or may
not result in a system with a common element of intersection.

If a sequence of shifting operations ends in a shifted system with
empty intersection, then we are certainly done.

Otherwise, some $\shift_{i_{0}\leftarrow j_{0}}$ results in a system
where every set contains $i_{0}$. Thus, before this step, we have
a system $\mathcal{F}$ where every set contains either $i_{0}$ or
$j_{0}$. Relabel $i_{0}$ to $1$ and $j_{0}$ to $2$, and continue
applying $\shift_{i\leftarrow j}$ operations over all $3\leq i<j$.
Thus, after these additional shift operations, we have $\left\{ 1,3,\dots,k+1\right\} $
and $\left\{ 2,3,\dots,k+1\right\} $ in the system. Without loss
of generality (since every set in $\mathcal{F}$ contains $1$ or
$2$), we also have all $k$-element subsets containing $\left\{ 1,2\right\} $;
otherwise, add them. Thus, we have $\bdry\left\{ 1,\dots,k+1\right\} $
contained in our system. As $\bdry\left\{ 1,\dots,k+1\right\} $ has
empty intersection and is preserved under all further shift operations
(over $1\leq i<j$), the result follows.
\end{proof}

\subsection*{Discussion}

In addition to being short and direct, our proof is relatively elementary,
using only shifting theory. Indeed, we recover a completely elementary
proof of the restriction of the Hilton-Milner Theorem to shifted systems.

A main difficulty in proofs of Hilton-Milner and/or Erd\H{o}s-Ko-Rado
type results is relating systems of $(k-1)$-element subsets to systems
of $k$-element subsets. Our approach handles this with the shadow
containment condition of Theorem~\ref{thm:MainTechnical}.

Our motivation here comes partly from combinatorial algebraic topology.
In particular, the simplicial complex generated by a shifted family
of $k$-element sets has homology with generators in $\mathcal{B}$
(using notation as in the proof of Theorem~\ref{thm:HM}). Thus,
Lemma~\ref{lem:Frankl-Furedi} transforms the combinatorial property
of empty intersection into a homological property. Kalai comments
on similar connections between intersection theorems and homology
in \cite[Section 6.4]{Kalai:2002}.

The approach also gives a unified proof of the well-known Erd\H{o}s-Ko-Rado
Theorem. More concretely, if we relax the hypothesis of Theorem~\ref{thm:MainTechnical}
to allow $\mathcal{B}$ to be empty, then the corresponding bound
is $\left|\mathcal{A}\right|+\left|\mathcal{B}\right|\leq{n-1 \choose k-1}$.
Erd\H{o}s-Ko-Rado now follows from replacing Theorem~\ref{thm:MainTechnical}
with the relaxed cross-intersection theorem in the proof of Theorem~\ref{thm:HM}.
The proof is similar to (and only slightly more complicated than)
the standard inductive proof of Erd\H{o}s-Ko-Rado for shifted systems.

The approach also recovers uniqueness of the largest family for Theorem~\ref{thm:HM}
when $n/2>k\geq4$. Here, we strengthen the hypothesis of Theorem~\ref{thm:MainTechnical}
to require $\mathcal{B}$ to have at least two elements. We discuss
the details in the following section.

\subsection*{Uniqueness of the Hilton-Milner family}

As mentioned in the discussion, the same techniques give uniqueness
of the maximum family in Theorem~\ref{thm:HM}. We prove:
\begin{thm}
\label{thm:StrictHM}In the situation of Theorem~\ref{thm:HM}, if
$4\leq k<n/2$ and $\left|\mathcal{F}\right|$ achieves the upper
bound, then there is some $k$-set $B$ and $i\notin B$ so that $\mathcal{F}$
consists of $B$ together with all $k$-sets that both contain $i$
and intersect $B$.
\end{thm}

We require $k\geq4$ in order to avoid some technicalities. In particular,
there is another family achieving the bound for $k=3$. See \cite{Hurlbert/Kamat:2018}
for more details and a different argument.

As in the proof of Theorem~\ref{thm:HM}, we reduce to a shifted
family, and prove for a shifted family.

The proof for a shifted family requires a completely straightforward
modification of Theorem~\ref{thm:MainTechnical}. We obviously require
$k\geq4.$ We also strengthen the hypothesis to require $\left|\mathcal{B}\right|\geq2$,
replacing the condition that $\left|\mathcal{B}\right|\geq1$; with
the strengthened hypothesis, the inequality is strict. Then in the
proof, we may have $\left|\mathcal{B}(n)\right|$ empty or nonempty.
If empty, then since $\mathcal{B}$ has at least two elements, so
$\mathcal{A}(n)$ is strictly smaller than the given bound. If nonempty,
then the bound in (\ref{eq:BnNonempty}) is already strict so long
as $k\geq4$. In either case, the induction step yields a strict inequality.

Theorem~\ref{thm:StrictHM} follows for shifted families by applying
the variant of Theorem~\ref{thm:MainTechnical} with $\left|\mathcal{B}\right|\geq2$
to the same families as in the proof of Theorem~\ref{thm:HM}.

It remains only to reduce to shifted families. This reduction requires
a bit of care. We did not find the following lemma in the literature,
although we believe it to be known to experts in the field.
\begin{lem}
Let $\mathcal{F}$ be a family of pairwise-intersecting $k$-element
subsets of $[n]$ with the additional property that for any $F_{0}\in\mathcal{F}$,
the intersection $\bigcap_{\mathcal{F}\setminus\{F_{0}\}}F$ is empty.
Then there is a shifted family $\mathcal{F}'$ satisfying the same
properties and with $\left|\mathcal{F}'\right|\geq\left|\mathcal{F}\right|$.
\end{lem}

\begin{proof}[Proof.]
 By the \emph{standard family}, we mean the shifted family with $A=\left\{ 2,\dots,k+1\right\} $,
$A'=\left\{ 2,\dots,k,k+2\right\} $, and all $k$-element sets that
both contain $1$ and intersect $A$ and $A'$. It is obvious that
the standard family is at least as large as any family where all but
two sets contain $1$.

Given $\mathcal{F}$, we perform a sequence of shifts. If these terminate
in a shifted family with the desired properties, then we are done.
Otherwise, an operation results in a family without the additional
property. Stopping just before this operation and relabeling elements,
we have a family containing sets with $1$ and not $2$, with $2$
and not $1$, with both $1$ and $2$, and possibly the set $B=\{3,\dots,k+2\}$.

We may assume without loss of generality that we have all sets containing
both $1,2$ and intersecting with $B$. Since these sets do not have
any common intersection other than $1,2$, the operations $\shift_{i\leftarrow j}$
over all $3\leq i<j$ preserve the additional property.

After shifting over $3\leq i<j$, if we have only one set with $1$
and not $2$, or only one set with $2$ and not $1$, then we replace
with the standard family. Otherwise, we have in the family $\{a,3,\dots,k+1\}$
and $\{a,3,\dots,k,k+2\}$ for $a=1,2$, along with all sets containing
$\{1,2\}$ and intersecting $B$. In particular, the family contains
as subfamilies both $\bdry\left\{ 1,\dots,k+1\right\} $ and $\bdry\left\{ 1,\dots,k,k+2\right\} $.
Both subfamilies have empty intersection and are preserved under all
shift operations, so we can now shift until the system stabilizes.
\end{proof}

\subsection*{Acknowledgements}

We particularly thank Dániel Gerbner for several helpful comments
about preprints of the paper. We also thank Peter Frankl, Balázs Patkós,
John Shareshian, and Tamás Sz\H{o}nyi.

\bibliographystyle{8_Users_russw_Documents_Research_mypapers_A_short_proof_of_the_Hilton-Milner_theorem_hamsplain}
\bibliography{7_Users_russw_Documents_Research_mypapers_A_short_proof_of_the_Hilton-Milner_theorem_Master}

\def\cprime{$'$}
\providecommand{\bysame}{\leavevmode\hbox to3em{\hrulefill}\thinspace}
\providecommand{\href}[2]{#2}
\begin{thebibliography}{10}

\bibitem{Frankl:1987}
Peter Frankl, \emph{The shifting technique in extremal set theory}, Surveys in
  combinatorics 1987 ({N}ew {C}ross, 1987), London Math. Soc. Lecture Note
  Ser., vol. 123, Cambridge Univ. Press, Cambridge, 1987, pp.~81--110.

\bibitem{Frankl:1991}
\bysame, \emph{Shadows and shifting}, Graphs Combin. \textbf{7} (1991), no.~1,
  23--29.

\bibitem{Frankl:2016}
\bysame, \emph{New inequalities for cross-intersecting families}, Mosc. J.
  Comb. Number Theory \textbf{6} (2016), no.~4, 27--32.

\bibitem{Frankl:2019}
\bysame, \emph{A simple proof of the {H}ilton-{M}ilner theorem}, Mosc. J. Comb.
  Number Theory \textbf{8} (2019), no.~2, 97--101.

\bibitem{Frankl/Furedi:1986}
Peter Frankl and Zolt\'{a}n F\"{u}redi, \emph{Nontrivial intersecting
  families}, J. Combin. Theory Ser. A \textbf{41} (1986), no.~1, 150--153.

\bibitem{Frankl/Tokushige:1992}
Peter Frankl and Norihide Tokushige, \emph{Some best possible inequalities
  concerning cross-intersecting families}, J. Combin. Theory Ser. A \textbf{61}
  (1992), no.~1, 87--97.

\bibitem{Gerbner/Patkos:2019}
D\'{a}niel Gerbner and Bal\'{a}zs Patk\'{o}s, \emph{Extremal finite set
  theory}, Discrete Mathematics and its Applications (Boca Raton), CRC Press,
  Boca Raton, FL, 2019.

\bibitem{Herzog/Hibi:2011}
J{\"u}rgen Herzog and Takayuki Hibi, \emph{Monomial ideals}, Graduate Texts in
  Mathematics, vol. 260, Springer-Verlag London Ltd., London, 2011.

\bibitem{Hilton/Milner:1967}
A.~J.~W. Hilton and E.~C. Milner, \emph{Some intersection theorems for systems
  of finite sets}, Quart. J. Math. Oxford Ser. (2) \textbf{18} (1967),
  369--384.

\bibitem{Hurlbert/Kamat:2018}
Glenn Hurlbert and Vikram Kamat, \emph{New injective proofs of the {E}rd{\H
  o}s-{K}o-{R}ado and {H}ilton-{M}ilner theorems}, Discrete Math. \textbf{341}
  (2018), no.~6, 1749--1754, {arXiv:1609.04714}.

\bibitem{Kalai:2002}
Gil Kalai, \emph{Algebraic shifting}, Computational commutative algebra and
  combinatorics ({O}saka, 1999), Adv. Stud. Pure Math., vol.~33, Math. Soc.
  Japan, Tokyo, 2002, pp.~121--163.

\bibitem{Wu/Li/Feng/Liu/Yu:2026}
Yongjiang Wu, Yongtao Li, Lihua Feng, Jiuqiang Liu, and Guihai Yu,
  \emph{Maximal intersecting families revisited}, Discrete Math. \textbf{349}
  (2026), no.~1, Paper No. 114654, 18, {arXiv:2411.03674}.

\end{thebibliography}

\end{document}